\documentclass[12pt,a4paper]{article}

\usepackage{amsmath}
\usepackage{amssymb}
\usepackage{amsthm}
\usepackage{euscript}
\usepackage[left=2 cm,right=2 cm,top=2.5 cm,bottom=2.5 cm]{geometry}
\usepackage{latexsym}

\date{}

\newtheorem{theorem}{Theorem}[section]
\newtheorem{lemma}[theorem]{Lemma}
\newtheorem{example}[theorem]{Example}
\newtheorem{remark}[theorem]{Remark}
\newtheorem{definition}[theorem]{Definition}

\begin{document}

\title{Remarks on mass transportation minimizing expectation of a minimum of affine functions}
\author{Alexander V. Kolesnikov and Nikolay Lysenko
\footnote{{
This study was supported by the RFBR project 17-01-00662 and the DFG project RO 1195/12-1.
The article was prepared within the framework of the Academic Fund Program at the National Research University
Higher School of Economics (HSE) in 2017--2018 (Grant No 17-01-0102) and by the Russian Academic Excellence Project ``5-100''.
}}}

\maketitle

\begin{abstract}
We study the Monge--Kantorovich problem with one-dimensional marginals $\mu$ and $\nu$ and 
the cost function $c = \min\{l_1, \ldots, l_n\}$
that equals the minimum of a finite number $n$ of affine functions $l_i$ 
satisfying certain non-degeneracy assumptions. We prove that the problem
is equivalent to a finite-dimensional extremal problem.  More precisely, it is shown that the solution is concentrated
on the union of $n$ products $I_i \times J_i$, where $\{I_i\}$ and $\{J_i\}$
are partitions of the real line into unions of disjoint connected sets.
The families of sets $\{I_i\}$ and $\{J_i\}$ have the following properties:   1)  $c=l_i$ on $I_i \times J_i$,
2) $\{I_i\}, \{J_i\}$ is  a couple of partitions solving an auxiliary $n$-dimensional extremal problem.
The result is partially generalized to the case of more than two marginals.

\vspace{5pt}

\end{abstract}

\section{Introduction}

Suppose we are given a couple of probability distributions $\mu, \nu$
on the real line that are assumed to be atomless and a Borel function $c \colon \mathbb{R}^2 \to \mathbb{R}$.
Denote by $\Pi(\mu,\nu)$ the set of Borel probability measures on $\mathbb{R} \times \mathbb{R}$
with marginals $\mu,\nu$.
Recall (see, e.g., \cite{BK}, \cite{B}, and~\cite{V})
that a measure $\pi\in \Pi(\mu,\nu)$ is a solution to the Monge--Kantorovich problem
if it gives the minimum to the functional $$\pi \mapsto \int c d \pi$$ on $\Pi(\mu,\nu)$:
\begin{equation}
\label{MK}
 \int c(x,y) d \pi \to \min, \ \pi \in \Pi(\mu,\nu).
\end{equation}

It is a classical and well-known fact that for a broad class of convex functions, such as, for instance,
$c=h(|x-y|)$ with a strictly convex function $h$, any solution to (\ref{MK}) is concentrated on the graph
of a non-decreasing function. The assumption of convexity of $c$ is standard
for many core results of the transportation theory. The case of the quadratic cost function $c=|x-y|^2$
is of particular interest.

In general, the Monge--Kantorovich problem with a concave cost $c$
is harder. Remarkably, solutions to (\ref{MK}) with concave $c$ have 
a completely different structure as compared to solutions for convex costs.
For instance, the corresponding optimal transportation mapping need not exist even for a strictly concave cost $c$.
The case of $c=h(|x-y|)$ with a strictly concave function $h$ has been studied in \cite{GangboMcCann},
where a general result on the existence of optimal transportation mappings has 
been established (see more recent developments in \cite{PPS}).
An exact solution in the one-dimensional case for $c=h(|x-y|)$ has been obtained in \cite{McCann}.
An algorithm to solve the discrete transportation problem with a concave cost has been proposed in \cite{DSS}.

We study problem (\ref{MK}) for the cost
$$
c = \min\{l_1, \l_2, \ldots, l_n\},
$$
$$
l_i = a_i x + b_i y + c_i.
$$
This problem, yet quite specific, is of particular interest, since the minima of affine functions 
are dense in the set of all concave functions.
It turns out that solutions have a  nice and relatively simple structure provided that 
the functions $l_i$ satisfy certain non-degeneracy assumption.
 In particular, problem (\ref{MK}) can be reduced to a finite-dimensional
optimization problem. The corresponding optimal transportation problem admits a non-unique solution.
Some of our results are generalized  for the case of $m \ge 2$ one-dimensional marginals. 
In particular, we find a complete  characterization
of the solution for the cost function
$$
c = \min\{x_1, x_2, \ldots, x_m\}
$$
(the minimum of coordinate functions).

We emphasize that nowadays the multi-marginal transportation problem (in particular, with one-dimensional marginals) 
is attracting  attention of many researchers
(see the recent survey \cite{Pass} about general results and particular examples). 
Both the concave and the multi-marginal transportation problems
have potential applications in economics (see \cite{McCann} and \cite{Pass}).

\section{Results}

\begin{definition}
We say that a couple of distinct affine functions $l_1, l_2$
satisfies the non-degeneracy assumption {\bf (A)} if
the set
$$
\Gamma_{1,2} = \{l_1=l_2\} \ne \emptyset
$$
is not  parallel to one of the axes.
\end{definition}

\begin{example}
The assumption {\bf(A)} is violated for
$c(x,y)=\min(x,x+y)$. Since $c(x,y) = x+\min(0,y)$,
 the corresponding problem is degenerate and every $\pi \in \Pi(\mu,\nu)$
is optimal.
\end{example}

\begin{definition}
Let $l_1, l_2$ be a couple of affine functions satisfying {\bf(A)} and let
$M = (x_0,y_0)$ be a point that belongs to $\Gamma_{1,2} = \{l_1=l_2\}$.
Let $Q_{M,l_1,l_2}$ be that one of the sets
$$
\Bigl\{ x \le x_0, y \ge y_0 \Bigr\} \cup \Bigl\{ x \ge x_0, y \le y_0 \Bigr\},
$$
$$
\Bigl\{ x \le x_0, y \le y_0 \Bigr\} \cup \Bigl\{ x \ge x_0, y \ge y_0 \Bigr\}
$$
which does not contain $\Gamma_{1,2}$. More precisely,  $Q_{M,l_1,l_2}$ is defined by the following condition:
$$
\Gamma_{1,2} \cap Q_{M,l_1,l_2} = M.
 $$
\end{definition}

\begin{definition}
Given a Borel cost function  $c\colon \mathbb{R} \times \mathbb{R} \to \mathbb{R} \cup \{\infty\}$,
 we say that  a subset $S \subset \mathbb{R} \times \mathbb{R}$ is  
 $c$-cyclically monotone(or simply cyclically monotone) if, for every non-empty sequence of 
 its elements $(x_1, y_1),\ldots, (x_n, y_n)$, the following inequality holds:
\begin{equation}
c(x_1, y_1) + c(x_2, y_2) + \cdots + c(x_n, y_n) \le c(x_1, y_n) + c(x_2, y_1) + \cdots + c(x_n, y_{n-1}).
\end{equation}
\end{definition}

It is known that for a broad class of cost functions any solution $\pi$ to (\ref{MK})
satisfies $\pi(S)=1$ for some cyclically monotone set $S$.

The following lemma is a version of the so-called  ``no-crossing rule'' (see \cite{McCann}).

\begin{lemma}
\label{l12lemma}
Let
$
l_1, l_2
$ be affine functions satisfying {\bf(A)} and let
$\pi$ be a solution to the Monge--Kantorovich problem (\ref{MK})
with
$$
c = \min(l_1, l_2).
$$
Then there exists a point $M \in \{l_1=l_2\}$ such that
the support of $\pi$ is contained in
$
Q_{M,l_1,l_2}.
$
\end{lemma}
\begin{proof}
It is clear that shifting the coordinates $x \to x - x_0, y \to y - y_0$ we can deal from the very beginning with
linear functions $l_1, l_2$. In particular, the origin belongs to $\Gamma_{1,2}$. In addition, since the marginals are fixed,
the assertion is invariant with respect to subtracting a linear function $l$: in place of $c$ one can deal with
$$
 c - l = \min(l_1 - l, l_2 -l).
$$
Passing to this situation if necessary, we can deal with the case
$c = \min(ax,by)$ for some $a \ne 0, b \ne 0$. Multiplying by a constant, we reduce the problem
to the case $c = \min(x, by)$. Let $b>0$ ($b<0$ can be considered similarly).

Note that every two-points set $\{(x_1,y_1), (x_2,y_2)\}$
satisfying
$$
\max(x_1, b y_1) < \min(x_2, b y_2)
$$
is not cyclically monotone.
Indeed, this can be easily verified by direct computations:
$$
\min(x_1,b y_1) + \min(x_2,by_2)  > x_1 + \min(x_2, by_1) \ge \min(x_1,b y_2) + \min(x_2, by_1).
$$
Now let us find the smallest number $s$ such that
\begin{equation}
\label{spoint}
\mu(-\infty,s) = \nu(s/b, + \infty).
\end{equation}
We claim that the measure $\pi$ is concentrated  on the set
$$Q_{(s,bs), x, by} = \Bigl\{ x \le s, y \ge s/b \Bigr\} \cup \Bigl\{ x \ge s, y \le s/b \Bigr\}.$$
Assume the contrary and find a point $(x_1,y_1)$ from the support of $\pi$ such that, say,
$$
x_1 <s, y_1< s/b.
$$
Then (\ref{spoint})  implies that there exists another point $(x_2,y_2)$ from the support of $\pi$ such that
$$
x_2 >s, y_2 > s/b.
$$
But this means that $\max(x_1, b y_2) < s < \min(x_2, by_2)$, hence $c$-monotonicity is violated.
\end{proof}

\begin{definition}
We say that we are given a $(\mu,\nu)$-partition of order $n$  if
\begin{enumerate}
\item
The $x$-{axis} and $y$-axis are represented as  unions of $n$ non-empty disjoint connected sets
$$\mathbb{R} \times \{0\} = I_1 \cup I_2 \ldots \cup I_n,$$
$$\{0\} \times \mathbb{R} = J_1 \cup J_2 \ldots \cup J_n.$$
\item
$$
\mu(I_i) = \nu(J_i)>0, \ \forall i \in \{1, 2, \ldots, n\}.
$$
\end{enumerate}
\end{definition}

Let us proceed to our first main result. 

\begin{theorem}
\label{partit-theorem}
Let $l_1, \ldots, l_n$ be $n$ affine functions
such that every two of them satisfy assumption {\bf(A)},
and, moreover, every set $A_i = \{c=l_i \}$, $1 \le i \le n$ has a non-empty interior, where
$$
c = \min_{ 1 \le i \le n} (l_1, l_2, \ldots,l_n)
$$
is the corresponding cost function.

Then, for every solution $\pi$ to the Monge--Kantorovich problem (\ref{MK}),
there exists a $(\mu,\nu)$-partition of order $k \le n$  such that
\begin{itemize}
\item[1)]
$\pi$ is concentrated on $\cup_{i=1}^k I_i \times J_i$,
\item[2)]
for every $ 1 \le i \le k$
either $\pi(I_i \times J_i)=0$
or (after a suitable renumeration of the functions $l_i$)
$$
c = l_i
$$
on $I_i \times J_i$.
\end{itemize}
\end{theorem}
\begin{proof}
Let us fix $i$ and consider the set
$$
\Omega_i = \{c=l_i\}.
$$
Assume that $\pi(\Omega_i)>0$.
Then for every $j \ne i$ set
$$
h_{ij} = \min(l_i,l_j)
$$
and consider the restrictions $\pi_{ij} = \pi|_{\Omega_{ij}}$
to the set
$$
\Omega_{ij} = \{c=h\}.
$$
We claim that  $\pi_{ij}$ is optimal for the cost function $h_{ij}$
and the projections $\pi_{ij} \circ Pr^{-1}_X$, $\pi_{ij} \circ Pr^{-1}_Y$ onto the axes.
Indeed, assuming the contrary consider another measure
$$
\tilde{\pi}_{ij} = \pi|_{\Omega_{ij}^c} + \hat{\pi}_{ij},
$$
where $\hat{\pi}_{ij}$ is optimal for $h_{ij}$ and $\pi_{ij} \circ Pr^{-1}_X$, $\pi_{ij} \circ Pr^{-1}_Y$.
Using that $c \le h_{ij}$, $c=h_{ij}$ on $\Omega_{ij}$ and $\hat{\pi}_{ij}$ is optimal, we obtain 
$$
\int c d \tilde{\pi}_{ij} = \int c d \pi|_{\Omega_{ij}^c}  + \int c d  \hat{\pi}_{ij}
\le \int c d \pi|_{\Omega_{ij}^c}  + \int h_{ij} d  \hat{\pi}_{ij}
< \int c d \pi|_{\Omega_{ij}^c}  + \int h_{ij} d  {\pi}_{ij} = \int c d \pi.
$$
This contradicts the optimality of $\pi$.

Applying Lemma \ref{l12lemma} we obtain that the supports 
of $\pi|_{\Omega_i} = \pi_{ij}|_{\Omega_i}$ and $\pi|_{\Omega_j} = \pi_{ij}|_{\Omega_j}$
are contained in the sets $L^1_{ij} \times M^{1}_{ij}$, $L^{2}_{ij} \times M^{2}_{ij}$, 
respectively, where $L^1_{ij}$ and $L^2_{ij} = \mathbb{R} \setminus L^1_{ij}$ are disjoint 
and connected (the same is true for $M^1_{ij}, M^2_{ij}$).
 The $i$th intervals of the desired $(\mu,\nu)$-partition are defined as follows:
 $$
 I_i = \cap_{j \ne i} L^{1}_{ij}, \  \ J_i = \cap_{j \ne i} M^{1}_{ij}.
 $$
 By construction
 $l_i = c$ on $I_i \times J_i$, hence $I_i \times J_i \subset \Omega_i$,  
 and the support of $\pi|_{\Omega_i}$ is contained in $I_i \times J_i$. The proof is complete.
\end{proof}

\begin{remark}
Let $\pi$ be a solution to the Monge--Kantorovich problem with the cost function $c$ satisfying 
assumptions of  Theorem \ref{partit-theorem}
and let $\{ I_i, J_i\}$ be the corresponding $(\mu,\nu)$-partition.
Then every measure with marginals $\mu,\nu$ concentrated 
on $\cup_{i=1}^k I_i \times J_i$ solves the same Monge--Kantorovich problem.

Moreover, if such a measure is concentrated 
on the graph of a mapping $T$, then $T$ is the corresponding optimal transportation.
\end{remark}

Theorem \ref{partit-theorem} shows, in particular, that the transportation problem 
is reduced to a finite-dimensional  problem of finding an optimal partition 
with the constraint  $c=l_i$ on $I_i \times J_i$.
In Theorem \ref{computation-theorem} below we present yet another  equivalent 
finite-dimensional problem, where we relax  the latter constraint on partitions and
replace integrals over minima by minima of certain (easy computable)  integrals. 
This viewpoint might be  useful for computational purposes.

\begin{theorem}
\label{computation-theorem}
Let $l_1, \ldots, l_n$ be  affine functions of the form 
$$
l_j = a_j x + b_j y + c_j
$$
satisfying the assumptions of Theorem \ref{partit-theorem}
and let
$$
c = \min_{ 1 \le i \le n} (l_1, l_2, \ldots,l_n).
$$
For every $(\mu,\nu)$-partition $\mathcal{P}$, we define a functional $J$ in the following way:
$$
J(\mathcal{P}) = \sum_i \min_{j} \Bigl( a_j \int_{I_i} x d \mu + b_j \int_{J_i} y d \nu  + c_j \mu(I_i) \Bigr).
$$
Then the minimal value of the functional
$J$ over  all $(\mu,\nu)$-partitions of order not greater that $n$ coincides with the minimum $K(\mu,\nu)$ 
of the Kantorovich functional for the cost
function $c$.
\end{theorem}
\begin{remark}
Note that
\begin{equation}
\label{JP}
J(\mathcal{P}) = \sum_i \min_j \int_{I_i \times J_i} l_j d \pi .
\end{equation}
for every $\pi$ with marginals $\mu,\nu$ and with the support in $\cup_{i=1}^k I_i \times J_i$.
In particular,
$$
J(\mathcal{P}) =   \sum_i \min_{j} \frac{1}{\mu(I_i)} \int_{I_i \times J_i}  l_j \ d  \mu|_{I_i} \times  \nu|_{J_i}.
$$
\end{remark}
\begin{proof}
It follows from representation (\ref{JP}) that
$$
J(\mathcal{P})  \ge \sum_i  \int_{I_i \times J_i} \min_j l_j d \pi  = \int c \ d \pi
$$
for every partition $\mathcal{P}$ and every measure $\pi$ with marginals $\mu, \nu$.
Thus,
$$
J(\mathcal{P})  \ge K(\mu,\nu).
$$
On the other hand, given a solution $\pi$ to the Monge--Kantorovich problem,
one can consider the particular partition $\mathcal{P}_0$ with the properties established in Theorem \ref{partit-theorem}.
Using that $l_i  = c$ on $I_i \times J_i$ for every $i$, one obtains
$$
J(\mathcal{P}_0)
=  \sum_i \min_j   \int_{I_i \times J_i} l_j d \pi = \sum_i    \int_{I_i \times J_i} l_i d \pi =  \sum_i   \int_{I_i \times J_i} c d \pi = \int c d \pi = K(\mu,\nu).
$$
The proof is complete.
\end{proof}

\begin{example}
Let $\mu$ and $\nu$ be the same Lebesgue measure on $[0,1]$.
Consider the set $\Pi$ of couples of partitions $(\mathcal{P}_x, \mathcal{P}_y)$ of $[0,1]$ of the form 
$$
I_i = [t_{i-1},t_i), \ J_i = [s_{i-1},s_i)
$$
 with the property
$$
s_{i}-s_{i-1} =
t_{i}-t_{i-1}.
$$
Then according to Theorem \ref{computation-theorem} the value of the Kantorovich functional equals
$$
 \min_{(\mathcal{P}_x, \mathcal{P}_y) \in \Pi} \sum_i (t_{i}-t_{i-1}) \min_{j}  \Bigl( a_j \frac{t_i + t_{i-1}}{2}+ b_j \frac{s_i + s_{i-1}}{2}  + c_j \Bigr).
$$
\end{example}

Finally, let us make some remarks about the multi-marginal case.
We give below a generalization of our main result for the case when
the number of affine functions coincides with the number of marginals. This covers, in particular, the cost function
$$
c(x_1, x_2, \ldots, x_n)=\min\{x_1, x_2, \ldots, x_n\}.
$$
We omit the proofs because they are completely similar to the case of two marginals.

\begin{remark}
From the description of  solutions to the Monge--Kantorovich problem for this cost function one can conclude 
that the straightforward generalization of Theorem \ref{partit-theorem} fails at least in the following respect: 
the projections of $\mbox{supp}(\pi) \cap \{c=l_i\}$ can have intersections for different $i$.
\end{remark}

\begin{definition}
We say that an $n$-tuple of distinct affine functions of $n$ arguments $l_1, \ldots, l_n$ 
satisfies the non-degeneracy assumption if the set
$$\Gamma_{1, \ldots, n} := \{l_1 = l_2 = \cdots = l_n\}$$ is a straight line 
spanned by an $n$-dimensional vector that has no zero components.
\end{definition}

\begin{definition}
Let $l_1, \ldots, l_n$ be an $n$-tuple of distinct affine functions satisfying the non-degeneracy assumption. 
Suppose that $M = (x_1^0, \ldots, x_n^0)$ is a point from
$$\Gamma_{l_1, \ldots, l_n} := \{l_1 = l_2 = \cdots = l_n\}.$$ 
Let $\EuScript{S}$ be the set $\{\le , \ge\}^n$, i.e., the set of sequences of $n$ symbols $''\le''$ or $''\ge''$.
We shall agree that $-''\le''$ coincides with  $''\ge''$ and $-''\ge''$ coincides with $''\le''$.

 Finally, for any $s \in \EuScript{S}$, let us define $Q_s$ as follows:
\begin{equation*}
Q_s = \bigcup\limits_{i \in \{1, \ldots, n\}} \{x_1 \, s[1] \, x_1^0, x_2 \, s[2] \, x_2^0, \ldots, x_i \, -\!\!s[i] 
\,  x_i^0, \ldots, x_n \, s[n] \, x_n^0\}.
\end{equation*}
\end{definition}

\begin{definition}
Take a directing vector $v$ of $\Gamma_{l_1, \ldots, l_n}$ 
with $v_1 >0$. 
Define $t \in \EuScript{S}$ 
by the following rule: $t[i]=''\ge''$ if $v_i >0$  and $t[i] =''\le''$ if $v_i<0$. 
The set $Q_t$ is further referred to as $Q_{M, \, l_1, \ldots, l_n}$.
\end{definition}

\begin{lemma}\label{l1-ln}
Let $l_1, \ldots, l_n$ be an $n$-tuple of distinct affine functions satisfying 
the non-dege\-ne\-racy assumption and let $\pi$ be a solution to the Monge--Kantorovich problem with $n$ marginals 
and the cost function
\begin{equation*}
c = \min(l_1, \ldots, l_n).
\end{equation*}
Assume, in addition, that every set $\{c=l_i\}$ has a non-empty interior.
Then there exists a point $M \in \Gamma_{1, \ldots, n}$ such that the support of $\pi$ 
is contained in $Q_{M, \, l_1, \ldots, l_n}$.
\end{lemma}

Applying Lemma \ref{l1-ln}, we obtain the following result.

\begin{theorem}
Let $\mu_i$, $1 \le i \le n$, be atomless probability measures on the real line.
Define
$$
s= \sup\Bigl\{x\colon \sum_{i=1}^n \mu_i(-\infty,x] \le 1\Bigr\}.
$$
Then the measure $\pi \in \Pi(\mu_1, \ldots, \mu_n)$ solves the Monge--Kantorovich problem with marginals $\mu_i$
and the cost function
$$
c = \min(x_1, \ldots, x_n)
$$
if and only if every point $x \in\mbox{supp}(\pi) \subset \mathbb{R}^n$ satisfies the following conditions:
\begin{itemize}
\item[1)]
if $x_i \le s$, then $x_j \ge s$ for every $ j \ne i$,
\item[2)]
if $x_i \ge s$, then there exists $j \ne i$ such that $x_j \le s$.
\end{itemize}
\end{theorem}

\end{document}